%% file: LocallyAndPolarHarmonicMaassForms.tex
\documentclass[12pt]{amsart}

\input{pakete}
\input{befehle}

\begin{document}

\title[Locally and Polar Harmonic Maass Forms]{Locally and Polar Harmonic Maass Forms for Orthogonal Groups of Signature $(2, n)$}
\author{Paul Kiefer}
\address{Department of Mathematics, University of Antwerp, BE-2000 Antwerp, Belgium.}
\email{Paul.Kiefer@uantwerpen.be}

\thanks{The author was funded by the Research Foundation - Flanders (FWO) within the framework of the Odysseus program project number G0D9323N, and by the Deutsche Forschungsgemeinschaft (DFG, German Research Foundation) -- SFB-TRR 358/1 2023 -- 491392403.}

\begin{abstract}
We generalize the notions of locally and polar harmonic Maass forms to general orthogonal groups of signature $(2, n)$ with singularities along real analytic and algebraic cycles. We prove a current equation for locally harmonic Maass forms and recover the Fourier expansion of the Oda lift involving cycle integrals. Moreover, using the newly defined polar harmonic Maass forms, we prove that meromorphic modular forms with singularities along special divisors are orthogonal to cusp forms with respect to a regularized Petersson inner product. Using this machinery, we derive a duality theorem involving cycle integrals of meromorphic modular forms along real analytic cycles and cycle integrals of locally harmonic Maass forms along algebraic cycles.
\end{abstract}

\maketitle

\tableofcontents

\section{Introduction}

Locally harmonic Maass forms are a class of harmonic modular forms with a specific type of singularity. They were first introduced by Bringmann, Kane and Kohnen \cite{BringmannLocallyHarmonic} to study the Shintani lift and, in the weight $k = 1$ case, by \cite{Hoevel} to study the Shimura lift, later extended to arbitrary weight by \cite{CrawfordFunke}. A key example of locally harmonic Maass forms are certain $\xi$-preimages $\calF_{1 - k, D}$ of cusp forms of weight $2k$ for $D > 0$ introduced by \cite{Zagier}
\begin{align}
	f_{k, D}(Z) = \sum_{Q \in \calQ_D} \frac{1}{Q(1, Z)^k}, \qquad \qquad Z \in \IH, \label{eq:ZagierCuspForm}
\end{align}
where $\calQ_D$ is the set of integral binary quadratic forms of discriminant $D$. For negative $D < 0$, Bengoechea \cite{Paloma} extended the definition of $f_{k, D}$ to obtain meromorphic cusp forms with poles at CM-points. The corresponding $\xi$-preimages $\calF_{1 - k, D}$, termed polar harmonic Maass forms, were introduced in \cite{BringmannPolar}. An important application of the theory of locally and polar harmonic Maass forms is a rationality result for cycle integrals of the meromorphic modular forms $f_{k, D}$, see \cite{Loebrich}, \cite{Alfes}. Recently, this theory was extended to orthogonal modular forms in signature $(2, 2)$ by \cite{HilbertPaper}, where a rationality result for cycle integrals of meromorphic modular forms was established.

The aim of this paper is to generalize this theory to arbitrary signature $(2, n)$, providing the foundation to extend many of the results involving locally and polar harmonic Maass forms beyond previously studied cases. The special case of signature $(2, 1)$ then corresponds to the original definition in \cite{BringmannLocallyHarmonic}, \cite{BringmannPolar}, while signature $(2, 2)$ corresponds to the case investigated in \cite{HilbertPaper}. \\

We will now describe our results in more detail. Let $L$ be an even lattice of signature $(2, n)$ with corresponding quadratic form $q$ and bilinear form $(\cdot, \cdot)$. Choose $e \in L$ primitive isotropic and $e' \in L'$ isotropic with $(e, e') = 1$. Let $K = e^\perp \cap e'^\perp \cap L$ be the orthogonal complement of the plane spanned by $e$ and $e'$. Then the hermitian symmetric domain corresponding to the orthogonal group $O(L \otimes \IR)$ can be realized as
$$\IH_n^\pm = \{ Z = X + i Y \in K \otimes \IC \ \vert\ q(Y) > 0\}$$
and we have a natural action of $O(L \otimes \IR)$ on $\IH_n^\pm$. We denote a connected component of $\IH_n^{\pm}$ by $\IH_n$ and write $O^+(L \otimes \IR) \subseteq O(L \otimes \IR)$ for the index $2$ subgroup that preserves the component. Let $\Gamma \subseteq O^+(L) = O(L) \cap O^+(V(\IR))$ be a neat congruence subgroup. Extending Zagier's modular forms in equation \eqref{eq:ZagierCuspForm} to arbitrary signature, Oda \cite{Oda} introduced a family of orthogonal cusp forms for $\Gamma$ of weight $\kappa > n$ indexed by $\mu \in V = L \otimes \IQ$ with $q(\mu) > 0$ given by
$$\omega_\mu^{\cusp}(Z) = \sum_{\gamma \in \Gamma_\mu \bs \Gamma} \frac{1}{(\lambda, Z + e - q(Z) e')^\kappa} \bigg\vert_\kappa \gamma,$$
where $\Gamma_\mu$ is the stabilizer of $\mu$ in $\Gamma$ and $\vert_\kappa$ is the Petersson slash operator, see Section \ref{sec:OrthogonalMF}. The meromorphic counterpart for $\nu \in V$ with negative $q(\nu) < 0$ is given by
$$\omega_\nu^{\mero}(Z) = \sum_{\gamma \in \Gamma_\nu \bs \Gamma} \frac{1}{(\lambda, Z + e - q(Z) e')^\kappa} \bigg\vert_\kappa \gamma$$
and has singularities of order $\kappa$ along the $\Gamma$-translates of the special divisor
$$T_\nu = \{ Z \in \IH_n\ \vert\ (\nu, Z + e - q(Z) e') = 0 \}.$$
A fundamental tool in this paper is the $\xi_{-\kappa}$-operator, generalizing constructions from \cite{HilbertPaper}, \cite{BruinierFunke}, \cite{BruinierFunkeKudla}. It defines a map from differential forms with values in the line bundle $\calL_{-\kappa}$ of modular forms of weight $-\kappa$ to differential forms with values in the line bundle $\calL_{\kappa}$  of modular forms of weight $\kappa$. One of the main results of this paper is the following generalization of \cite[Proposition 6.1]{BringmannLocallyHarmonic},  \cite[Proposition 5.2]{HilbertPaper}.

\begin{thm}[see \thref{thm:LocallyHarmonicMF}]
There exists a real analytic $(n,n-1)$-form with values in $\calL_{-\kappa}$ denoted by $\Omega_{\mu}^{\cusp}$ that is harmonic outside a real analytic cycle $C_\mu$ such that $\xi_{-\kappa} \Omega_\mu^{\cusp} = \omega_\mu^{\cusp}$. The space generated by $\Omega_\mu^{\cusp}$ is called the space of locally harmonic Maass forms.
\end{thm}

These $\xi_{-\kappa}$-preimages can be considered as a distribution on $\calA^{n,n-1}(\Gamma \bs \IH_n, \calL_{-1})$ via
$$[\Omega_\mu^{\cusp}](H) = \int_{\Gamma \bs \IH_n} H(Z) \wedge \overline{*}_{-\kappa} \Omega_\mu^{\cusp}(Z)$$
and we define the distributional $\xi_{-\kappa}$-operator by
$$(\xi_{-\kappa}[\Omega_\mu^{\cusp}])(h) = -[\Omega_\mu^{\cusp}](\xi_\kappa h).$$
Using a Stokes-argument we prove the following current equation, which generalizes \cite[Theorem 8.3]{CrawfordFunke}, \cite[Theorem 8.10]{HilbertPaper}.

\begin{thm}[see \thref{thm:CurrentEq}]
	We have the current equation
	$$\xi_{-\kappa}[\Omega_\mu^{\cusp}] = [\omega_\mu^{\cusp}] + C_{n, \kappa} \delta_{C_\mu},$$
	where $C_{n, \kappa}$ is an explicit constant and
	$$\delta_{C_\mu}(h) = q(\mu)^{\frac{n}{2} - \kappa} \int_{\Gamma_\mu \bs C_\mu} h(Z) (\lambda, \psi(Z))^{\kappa - n} dZ$$
	is the cycle integral of $h$ along $C_\mu$.
\end{thm}

As an application we use this current equation to recover the Fourier expansion of the Oda lift, see \thref{cor:OdaLift} in Section \ref{sec:OdaLift}.

We want to mention that our definition of locally harmonic Maass form is slightly more restrictive than the original definition of \cite{BringmannLocallyHarmonic}, see \thref{rem:LocallyHarmonicMaassForm}.

Similarly, we introduce the notion of polar harmonic Maass forms, extending the definitions of \cite{BringmannLocallyHarmonic}, \cite[Section 5]{HilbertPaper} to general signature $(2, n)$ and prove the following

\begin{thm}[see \thref{thm:PolarHarmonicMF}]
	There exists a real analytic $(n,n-1)$-form with values in $\calL_{-\kappa}$ denoted by $\Omega_{\nu}^{\mero}$ that is harmonic outside the algebraic cycle $T_\nu$ such that $\xi_{-\kappa} \Omega_\nu^{\mero} = \omega_\nu^{\mero}$. The space generated by $\Omega_\nu^{\mero}$ is called the space of polar harmonic Maass forms.
\end{thm}

Again, this definition is more restrictive then the definition in \cite{BringmannPolar}, see \thref{rem:PolarHarmonicMaassForms}.

Using this preimage and a Stokes' argument we deduce

\begin{thm}[see \thref{thm:OrthogonalityOnCF}]
	Let $h : \IH_n \to \IC$ be a real analytic cusp form. Then
	$$\langle h, \omega_\nu^{\mero} \rangle_{\reg} = 0,$$
	where $\langle \cdot, \cdot \rangle_{\reg}$ is a regularized Petersson inner product, see \thref{def:RegularizedInnerProduct}.
\end{thm}

In particular, we obtain $\langle \omega_\mu^{\cusp}, \omega_\nu^{\mero} \rangle_{\reg} = 0$. Using the $\xi_{-\kappa}$-preimage of $\omega_\mu^{\cusp}$ instead, we obtain the following duality theorem, which generalizes \cite[Theorem 8.8, Theorem 8.9]{HilbertPaper}.

\begin{thm}[see \thref{thm:DualityTheorem}]
	Assume that the cycles $C_\mu$ and $T_\nu$ do not intersect. Then
	$$C_{n, \kappa} \delta_{C_\mu}(\omega_\nu^{\mero}) = \delta_{T_\nu}(\Omega_\mu^{\cusp}),$$
	where $\delta_{T_\nu}(\Omega_\mu^{\cusp})$ is a cycle integral over the algebraic cycle $T_\nu$, see Section \ref{sec:PolarHarmonicMF}.
\end{thm}

\begin{rem}
	Throughout the paper we assume that $\Gamma \subseteq O^+(L)$ is a neat congruence subgroup. For more general congruence subgroups $\Gamma \subseteq O^+(L)$, there exists a neat congruence subgroup $\Gamma' \subseteq \Gamma$ of finite index by \cite[p. 117]{Borel} and the treatment of $\Gamma$ can be reduced to $\Gamma'$. This leads to essentially the same formulas.
\end{rem}

\subsection{Outlook}

This work lays the foundation to generalize many of the results involving locally and polar harmonic Maass forms in signature $(2, 1)$ and signature $(2, 2)$ to arbitrary signature $(2, n)$.

In particular, in an upcoming work we will show that certain linear combinations of locally and polar harmonic Maass forms are theta lifts of Poincar\'e series, compare \cite[Theorem 7.3]{HilbertPaper}, \cite[Theorem 1.3(2)]{BringmannThetaLift}, \cite[Theorem 6.2]{BringmannPolar} and we expect that we can obtain similar rationality results for cycle integrals of meromorphic modular forms as in \cite{Alfes}, \cite{HilbertPaper}, \cite{Loebrich}.

Moreover, by \cite[Theorem 1]{Oda} (linear combinations of) the cycle integrals of $\omega_\mu^{\cusp}$ over real analytic cycles are the Fourier coefficients of an elliptic cusp form. In \cite{BringmannGenerating}, a generating series of locally harmonic Maass forms was completed to a modular form and in \cite{Winding}, the modularity of a generating series of cycle integrals of meromorphic differential forms was proven.
Analogously, we expect that (linear combinations of) the cycle integrals $\delta_{C_\mu}(\omega_\nu^{\mero})$ are the Fourier coefficients of a non-holomorphic modular form and that the non-holomorphic part is related to the cycles integrals $\delta_{T_\nu}(\Omega_\mu^{\cusp})$.

\subsection{Outline}

The paper is organized as follows. In Section \ref{sec:OrthogonalMF} we review the theory of orthogonal modular forms and recall the real analytic and algebraic cycles. In Section \ref{sec:HermitianGeometry} we introduce the $\xi_\kappa$-operator. In particular we recall the necessary hermitian geometry. We end the section by some tedious calculations that are needed to obtain the $\xi$-preimages of $\omega_\mu^{\cusp}$ and $\omega_\nu^{\mero}$. In Section \ref{sec:LocallyHarmonicMF} we define locally harmonic Maass forms and prove their current equation in \thref{thm:CurrentEq}. Moreover, we sketch how to recover the Fourier expansion of the Oda lift from the current equation in \thref{cor:OdaLift}. Finally, in Section \ref{sec:PolarHarmonicMF} we introduce polar harmonic Maass forms and define the regularized Petersson inner product. We then show that polar harmonic Maass forms are orthogonal to cusp forms in \thref{thm:OrthogonalityOnCF} and prove the duality \thref{thm:DualityTheorem}.

\subsection*{Acknowledgments}

I would like to thank Claudia Alfes for valuable discussions related to this work and many helpful comments on an earlier draft, which significantly improved this paper. Moreover, I thank R\i zacan \c{C}ilo\u{g}lu for some helpful comments on an earlier draft of this paper.

\section{Orthogonal Modular Forms}\label{sec:OrthogonalMF}

\subsection{Orthogonal Upper Half Plane}

Let $L$ be an even lattice of signature $(2, n), n \in \IN$ with quadratic form $q$ and bilinear form $(\cdot, \cdot)$. Let $e \in L$ be primitive isotropic and $e' \in L'$ with $(e, e') = 1$. Let $\tilde{e}' = e' - q(e') e$, so that $q(\tilde{e}') = 0$. Let $K = L \cap e^\perp \cap e'^\perp = L \cap e^\perp \cap \tilde{e}'^\perp$. We write $V = L \otimes \IQ,\ V(\IR) = L \otimes \IR,\ V(\IC) = L \otimes \IC$ and $W = K \otimes \IQ,\ W(\IR) = K \otimes \IR,\ W(\IC) = K \otimes \IC$. We define
$$\IH_n^{\pm} = \{ Z = X + i Y \in W(\IC)\ \vert\ q(Y) > 0 \}.$$
We have a map $\psi : \IH_n^{\pm} \to V(\IC)$ given by
$$\psi(Z) = Z - q(Z) e + \tilde{e}'.$$
This yields a biholomorphic map
$$\IH_n^{\pm} \to \{ [Z] = [X + iY] \in \IP(V(\IC))\ \vert\ (Z, Z) = 0, (Z, \overline{Z}) > 0 \}.$$
The orthogonal group $O(V(\IR))$ acts naturally on the right hand side, which yields an action on $\IH_n^{\pm}$ via $\sigma Z := \psi^{-1}(\sigma(\psi(Z)))$. By slight abuse of notation we write $\psi(X), \psi(Y)$ for the real and imaginary part of $\psi(X + iY)$.

We define the \emph{factor of automorphy} by
$$j(\sigma, Z) = (e, \sigma(\psi(Z)))$$
so that
$$j(\sigma, Z)\psi(\sigma Z) = \sigma(\psi(Z)).$$
In particular, we have for $\lambda \in V(\IR)$
\begin{align}
(\lambda, \psi(\gamma Z)) = \frac{(\gamma^{-1} \lambda, \psi(Z))}{j(\gamma, Z)}. \label{eq:qZModularity}
\end{align}
Moreover, an easy calculation yields
$$q(\Im(\gamma(Z))) = \frac{q(Y)}{\lvert j(\gamma, Z) \rvert^2}.$$
The space $\IH_n^\pm$ consists of two connected components and we let $\IH_n$ be one of them and $G = O^+(V(\IR)) \subseteq O(V(\IR))$ the index $2$ subgroup that acts on $\IH_n$. We call this the \emph{orthogonal upper half plane}.

Let
$$\Gr(L) = \Gr(V) = \{ (Z^+, Z^-)\ \vert Z^+, Z^- \subseteq V, Z^+ \oplus Z^- = V, Z^+ \perp Z^-, q \vert_{Z^+} > 0, q \vert_{Z^-} < 0 \}$$
be the associated \emph{Grassmannian}. We have an isomorphism
$$\IH_n \to \Gr(L), \qquad Z \mapsto (Z^+, Z^-) := (\langle \psi(X), \psi(Y) \rangle, \langle \psi(X), \psi(Y) \rangle^\perp).$$
For $\lambda \in V(\IR)$ and $Z \in \IH_n$ we write $\lambda_{Z^\pm}$ for the projection of $\lambda$ onto $Z^\pm \subseteq V(\IR)$.

\subsection{Real Analytic and Algebraic Cycles}

For $\mu \in V(\IR)$ with $q(\mu) > 0$ we consider the real analytic submanifold of $\IH_n$ given by
$$C_\mu = \{ Z \in \IH_n\ \vert\ q(\mu_{Z^-}) = 0 \}.$$
It corresponds to the set of $(Z^+, Z^-) \in \Gr(L)$ such that $\mu \in Z^+$ and thus to the Grassmannian of the real quadratic space $\mu^\perp$ of signature $(1, n)$. In particular, it is a hyperbolic space of dimension $n$.

Similarly, for $\nu \in V(\IR)$ with $q(\nu) < 0$ we consider the complex submanifold
$$T_\nu = \{ Z \in \IH_n\ \vert\ q(\nu_{Z^+}) = 0\},$$
which corresponds to the set of $(Z^+, Z^-) \in \Gr(L)$ with $\nu \in Z^-$ and thus corresponds to the Grassmannian of the real quadratic space $\nu^\perp$ of signature $(2, n-1)$, which is isomorphic to $\IH_{n-1}$ and has complex codimension $1$ in $\IH_n$.

We call a congruence subgroup $\Gamma \subseteq O^+(L) = O(L) \cap O^+(V(\IR))$ \emph{neat}, if for all $\gamma \in \Gamma$ the subgroup of $\IC^{\times}$ generated by the eigenvalues of $\gamma$ is torsion-free. From now on, let $\Gamma \subseteq O^+(L)$ always be a neat congruence subgroup. Let $\nu \in V, q(\nu) < 0$ and $\Gamma_\nu$ be the stabilizer of $\nu$ in $\Gamma$. Then $\Gamma_\nu$ acts on the cycle $T_\nu$ and $\Gamma_\nu \bs T_\nu$ defines an algebraic cycle in $\Gamma \bs \IH_n$. Analogously, if $\mu \in V,\ q(\mu) > 0$ and $\Gamma_\mu$ is the stabilizer of $\mu$ in $\Gamma$, then $\Gamma_\mu \bs C_\mu$ defines a real analytic cycle in $\Gamma \bs \IH_n$, see for example \cite[Theorem 4.1]{Funke}.

Since our later results are conditional on the assumption that the cycles $C_\mu$ and $T_\nu$ do not intersect, we mention the following

\begin{lem}
For $\mu, \nu \in V(\IR)$ with $q(\mu) > 0, q(\nu) < 0$. Then $C_\mu \cap T_\nu \neq \emptyset$ if and only if $(\mu, \nu) = 0$.
\end{lem}

\begin{proof}
This directly follows from the explicit description of the cycles. We have $Z \in C_\mu \cap T_\nu$ if and only if $\mu \in Z^+, \nu \in Z^-$. But this can only happen, if $\mu \perp \nu$, i.e. $(\mu, \nu) = 0$.
\end{proof}

\subsection{Orthogonal Modular Forms}

For a function $h : \IH_n \to \IC$ we define the \emph{weight $\kappa \in \IZ$ Petersson slash operator} by
$$(h \vert_\kappa \sigma)(Z) = j(\sigma, Z)^{-\kappa} h(\sigma Z), \qquad \sigma \in O^+(V(\IR)).$$
A holomorphic function $h : \IH_n \to \IC$ is called \emph{modular form of weight $\kappa$ with respect to $\Gamma \subseteq O^+(L)$}, if $h \vert_\kappa \sigma = h$ for all $\sigma \in \Gamma$. By the Koecher principle, a modular form is automatically holomorphic at the cusps. Otherwise we will add this as an assumption. If a modular form vanishes at all cusps, we call it a \emph{cusp form}.

We extend the weight $\kappa$ slash operator to differential forms in $\calA^{*, *}(\IH_n)$ in the obvious way. If $\Gamma$ is torsion-free (for example, if $\Gamma$ is neat), then modular forms are the global sections of a hermitian line bundle $\calL_\kappa$ on $\Gamma \bs \IH_n$ whose hermitian metric is given by the Petersson metric $h_1(Z) \overline{h_2(Z)} q(Y)^\kappa$ on the fibers and differential forms on $\Gamma \bs \IH_n$ with values in $\calL_\kappa$ are the same as differential forms on $\IH_n$ that are invariant under the weight $\kappa$ slash operator.

Let $\mu \in V$ with $q(\mu) > 0$ and $\kappa > n$. Following \cite{Oda} we define the orthogonal cusp form
$$\omega_{\mu, \kappa, \Gamma}^{\cusp}(Z) = \sum_{\gamma \in \Gamma_\mu \bs \Gamma} \frac{1}{(\mu, \psi(Z))^\kappa} \bigg\vert_\kappa \gamma.$$
This definition makes sense, since by \eqref{eq:qZModularity} we have
$$(\mu, \psi(\gamma Z)) = \frac{(\gamma^{-1} \mu, \psi(Z))}{j(\gamma, Z)} = \frac{(\mu, \psi(Z))}{j(\gamma, Z)}$$
for $\gamma \in \Gamma_\mu$. Analogously, for $\nu \in V$ with $q(\nu) < 0$ and $\kappa > n$ we define
$$\omega_{\nu, \kappa, \Gamma}^{\mero}(Z) = \sum_{\gamma \in \Gamma_\nu \bs \Gamma} \frac{1}{(\nu, \psi(Z))^\kappa} \bigg\vert_\kappa \gamma.$$
These define meromorphic modular forms with singularities along the $\Gamma$-translates of the algebraic cycle $T_\nu$ of order $\kappa$, but still vanish at the cusps. We will usually write $\omega_{\mu}^{\cusp}$ and $\omega_\nu^{\mero}$, if the weight $\kappa$ and the neat congruence subgroup $\Gamma$ are clear from the context. The following lemma is used to reduce many of the properties of $\omega_\mu^{\cusp}$ and $\omega_\nu^{\cusp}$ to specific choices of $\mu$ and $\nu$.
\begin{lem}\thlabel{lem:omegaProperties}
For $\lambda \in V$ with $q(\lambda) \neq 0$ let $\omega_{\lambda, \kappa, \Gamma}$ be the function $\omega_{\lambda, \kappa, \Gamma}^{\cusp}$, if $q(\lambda) > 0$ and $\omega_{\lambda, \kappa, \Gamma}^{\mero}$, if $q(\lambda) < 0$. Then
\begin{enumerate}[(i)]
	\item For $\lambda \in V$ and $r \in \IQ^\times$ we have $\omega_{r \lambda, \kappa, \Gamma} = r^{-\kappa} \omega_{\lambda, \kappa, \gamma}$.
	\item If $\lambda, \lambda' \in V$ and $\gamma \in \Gamma$ such that $\lambda = \gamma \lambda'$, then $\omega_{\lambda, \kappa, \Gamma} = \omega_{\lambda', \kappa, \Gamma}$.
	\item If $\gamma \in O^+(V(\IR))$, then $\omega_{\lambda, \kappa, \Gamma} \vert_\kappa \gamma = \omega_{\gamma^{-1}\lambda, \kappa, \gamma^{-1} \Gamma \gamma}$, where now $\gamma^{-1} \Gamma \gamma \subseteq O^+(\gamma^{-1} L)$.
\end{enumerate}
\end{lem}
The proof is straightforward. For
$$\omega^{\mero} = \sum_{\substack{\nu \in \Gamma \bs L' / \IQ^\times \\ q(\nu) < 0}} c(\nu) q(\nu)^{\kappa} \omega_\nu^{\mero},$$
we can thus define the associated divisor
$$T(\omega^{\mero}) = \sum_{\substack{\nu \in \Gamma \bs L' / \IQ^\times \\ q(\nu) < 0}} c(\nu) T_\nu.$$

\begin{rem}
	For $\lambda \in V \setminus \{0\}$ with $q(\lambda) = 0$ one can consider analogously defined series. It turns out that these are given by Eisenstein series as in \cite{Zagier}. See for example \cite{Kiefer}, \cite{KieferNonHolEisenstein}, \cite{Schaps}.
\end{rem}

\section{The $\xi$-Operator on the Orthogonal Upper Half Plane}\label{sec:HermitianGeometry}

In this section we introduce the necessary hermitian geometry on the orthogonal upper half plane to define the differential operator $\xi_{-\kappa}$ and construct $\xi_{-\kappa}$-preimages of $\omega_\nu^{\mero}, \omega_\mu^{\cusp}$. We follow the exposition of \cite[Section 5]{Kiefer}.

\subsection{Hermitian Geometry on the Orthogonal Upper Half Plane}

As before, let $e \in L$ be primitive isotropic, $e' \in L'$ with $(e, e') = 1$ and $\tilde{e}' = e' - q(e') e$. Let $K = L \cap e^\perp \cap e'^\perp$. Let $b_1, \ldots, b_n$ be an orthonormal basis of $W(\IR)$, i.e. $q(b_1) = 1,\ q(b_i) = -1,\ 2 \leq i \leq n$ and $(b_i, b_j) = 0$ for $i \neq j$. We will write $Z = \sum_{j = 1}^n z_j b_j$ for $Z \in \IH_n$. The $(1, 1)$-form
$$\omega = - \frac{i}{2} \partial \overline{\partial} \log(q(Y))$$
defines an $O^+(V(\IR))$-invariant K\"ahler form. We write
$$\omega = \frac{i}{2} \sum_{i,j} h_{ij}(Y) d z_i \wedge d \overline{z}_j,$$
and $h(Y) = (h_{ij}(Y))_{1 \leq i,j \leq n}$, so that
$$h(Y) = \frac{1}{q(Y)} \begin{pmatrix}y_1^2 & -y_1 y_2 & \hdots & \hdots & -y_1 y_n \\
-y_1 y_2 & y_2^2 & y_2 y_3 & \hdots & y_2 y_n \\
\vdots & y_2 y_3 & \ddots & \ddots & \vdots \\
\vdots & \vdots & \ddots & \ddots & y_{n-1} y_n \\
-y_1 y_n & \hdots & \hdots & y_{n-1} y_n & y_n^2 \end{pmatrix} + \frac{1}{2q(Y)} \begin{pmatrix} -1 & & &\\
& 1 & \\
& & \ddots & \\ 
& & & \ddots & \\
& & & & 1\end{pmatrix}$$
defines the corresponding hermitian metric on $\IH_n$. Its inverse $h(Y)^{-1} = (h^{ij}(Y))_{1 \leq i,j \leq n}$ is given by
$$h(Y)^{-1} = 4 \begin{pmatrix}y_1^2 & y_1 y_2 & \hdots & \hdots & y_1 y_n \\
	y_1 y_2 & y_2^2 & y_2 y_3 & \hdots & y_2 y_n \\
	\vdots & y_2 y_3 & \ddots & \ddots & \vdots \\
	\vdots & \vdots & \ddots & \ddots & y_{n-1} y_n \\
	y_1 y_n & \hdots & \hdots & y_{n-1} y_n & y_n^2 \end{pmatrix} + 2 q(Y) \begin{pmatrix} -1 & & &\\
	& 1 & \\
	& & \ddots & \\ 
	& & & \ddots & \\
	& & & & 1\end{pmatrix}.$$
We have $\det(h(Y)) = \frac{1}{2^n q(Y)^n}$, so the volume form is given by
$$d \mu(Z) = \frac{1}{(4i q(Y))^n} d z_1 \wedge d \overline{z}_1 \wedge \ldots \wedge d z_n \wedge d \overline{z}_n.$$
We write
$$\widehat{d z_i} = d z_1 \wedge d \overline{z}_1 \wedge \ldots d \overline{z}_{i-1} \wedge d \overline{z}_{i} \wedge \ldots \wedge d z_n \wedge d \overline{z}_n$$
and defining $\widehat{d \overline{z}_i}$ analogously, we obtain
\begin{align}
	d z_i \wedge \widehat{d z_i} = (4 i q(Y))^n d \mu(Z), \qquad d \overline{z}_i \wedge \widehat{d \overline{z}_i} = -(4 i q(Y))^n d \mu(Z). \label{eq:dzWedgeHatdz}
\end{align}
Moreover, we write $dZ = d z_1 \wedge \ldots d z_n$.

Let $\alpha, \beta \in \calA^{p,q}(\IH_n)$ be $(p,q)$-forms. Then the \emph{Hodge-$\overline{*}$-operator} is defined by the equality
$$\alpha \wedge \overline{*} \beta = \langle \alpha, \beta \rangle d \mu(Z),$$
where $\langle \cdot, \cdot \rangle$ is the inner product on differential forms induced by the hermitian metric. In particular, we have
\begin{align}
	\overline{*} d \overline{z}_i = -\frac{1}{2} \sum_{j = 1}^n \frac{h^{ij}(Y)}{(4i q(Y))^n} \widehat{d \overline{z}_j}. \label{eq:Stardz}
\end{align}

Now let $\Gamma \subseteq O^+(L)$ be a neat congruence subgroup. Then the quotient $\Gamma \bs \IH_n$ is a hermitian manifold with metric induced from the hermitian metric on $\IH_n$. Let $\calL_\kappa$ be the line bundle of modular forms of weight $\kappa$, so that its global sections $\calL_\kappa(\Gamma \bs \IH_n)$ are modular forms. Using the hermitian metric, the dual bundle $\calL_\kappa^*$ can be identified with $\calL_{-\kappa}$ and the mapping $F(Z) \mapsto \overline{F(Z)} q(Y)^\kappa$ defines an anti-linear bundle isomorphism $\calL_\kappa \to \calL_{-\kappa}$. This yields a \emph{Hodge-$\overline{*}_\kappa$-operator} on differential forms with values in $\calL_\kappa$ defined by
$$\overline{*}_\kappa : \calA^{p,q}(\Gamma \bs \IH_n, \calL_\kappa) \to \calA^{n-p, n-q}(\Gamma \bs \IH_n, \calL_{-\kappa}), \qquad \overline{*}_\kappa(\phi \otimes F) = \overline{*} \phi \otimes \overline{F(Z)}q(Y)^\kappa.$$
We obtain an inner product on $\calA^{p,q}(\Gamma \bs \IH_n, \calL_\kappa)$ via
$$\langle H_1, H_2 \rangle = \int_{\Gamma \bs \IH_n} H_1(Z) \wedge \overline{*}_\kappa H_2(Z).$$
As in \cite{HilbertPaper}, \cite{BruinierFunke}, \cite{BruinierFunkeKudla}, we define the \emph{$\xi_\kappa$-operator} by
$$\xi_\kappa = \overline{*}_\kappa \overline{\partial} : \calA^{p, q}(\Gamma \bs \IH_n, \calL_\kappa) \to \calA^{n - p, n - q -1}(\Gamma \bs \IH_n, \calL_{-\kappa}).$$
In particular, on $\calA^{n, n-1}(\Gamma \bs \IH_n, \calL_\kappa)$ we have $\xi_{-\kappa} \xi_\kappa = \Delta_\kappa$, where
$$\Delta_\kappa = \overline{*}_{-\kappa} \overline{\partial} \overline{*}_\kappa \overline{\partial} + \overline{\partial} \overline{*}_{-\kappa} \overline{\partial} \overline{*}_\kappa$$
is the \emph{weight $\kappa$ Laplace operator}. We say that a differential form $H$ with values in $\calL_\kappa$ is \emph{harmonic}, if $\Delta_\kappa H = 0$. We can also view these as differential operators on $\calA^{p, q}(\IH_n)$ and we will do this throughout. As in \cite[Lemma 4.2]{HilbertPaper}, an application of Stokes' theorem yields

\begin{lem}\thlabel{lem:Stokes}
Let $U \subseteq V \subseteq \Gamma \bs \IH_n$ be open with $\overline{U} \subseteq V$. Let $h_1, h_2 \in \calL_\kappa(V)$ be cuspidal and assume $h_2 = \xi_{-\kappa} H_2$ for some $H_2 \in \calA^{n,n-1}(V, \calL_{-\kappa})$. Then
$$\int_U h_1(Z) \overline{h_2(Z)} q(Y)^\kappa\ d \mu(Z) = -\int_U \xi_\kappa h_1(Z) \wedge \overline{*}_{-\kappa} H_2(Z) + \int_{\partial U} h_1(Z) H_2(Z).$$
\end{lem}

\begin{proof}
We have
$$\int_U h_1(Z) \overline{h_2(Z)} q(Y)^\kappa\ d \mu(Z) = \int_U h_1(Z) \wedge \overline{*}_\kappa \xi_{-\kappa} H_2(Z).$$
Using $\overline{*}_\kappa \overline{*}_{-\kappa} = \id$ on $2n$-forms yields $\overline{*}_\kappa \xi_{-\kappa} H_2(Z) = \overline{\partial} H_2(Z)$. Since $H_2 \in \calA^{n, n-1}(V, \calL_{-\kappa})$ we have
$$d (h_1(Z) H_2(Z)) = \overline{\partial} (h_1(Z) H_2(Z)) = \overline{\partial} h_1(Z) \wedge H_2(Z) + h_1(Z) \wedge \overline{\partial} H_2(Z).$$
Hence,
$$\int_U h_1(Z) \wedge \overline{*}_\kappa \xi_{-\kappa} H_2(Z) = -\int_U \overline{\partial} h_1(Z) \wedge H_2(Z) + \int_U d(h_1(Z) H_2(Z)).$$
Since $\overline{*}$ is an isometry, we obtain from the definition of the Hodge-$\overline{*}$-operator
$$\overline{\partial} h_1(Z) \wedge H_2(Z) = \left\langle \overline{\partial} h_1(Z), \overline{*}_{-\kappa} H_2(Z) \right\rangle\ d \mu(Z) = \xi_\kappa h_1(Z) \wedge \overline{*}_{-\kappa} H_2(Z)$$
and an application of Stokes' theorem yields the result.
\end{proof}

\subsection{Some Formulas Involving the $\xi$-Operator}

We will now state and prove some formulas that are needed to define the $\xi$-preimages of $\omega_{\mu}^{\cusp},\ \omega_\nu^{\mero}$. Recall that we have chosen a basis $b_1, \ldots, b_n$ of $W(\IR)$ with $q(b_1) = 1$ and $q(b_j) = -1$ for $2 \leq j \leq n$. Let $\epsilon_i = 1$ for $i = 1$ and $\epsilon_i = -1$ for $i \neq 1$. For $\lambda \in V(\IR)$ write $\lambda = \lambda_e e + \lambda_e' \tilde{e}' + \sum_{j = 1}^n \lambda_j b_j$. Recall that
$$\psi(Z) = Z - q(Z) e + \tilde{e}'.$$
Therefore,
$$(\lambda, \psi(Z)) = \lambda_e - \lambda_e' q(Z) + 2 \sum_{j = 1}^n \epsilon_j \lambda_j z_j, \qquad q(Z) = \sum_{j = 1}^n \epsilon_j z_j^2$$
and
$$\overline{\partial} (\lambda, \psi(\overline{Z})) = 2 \sum_{j = 1}^n \epsilon_j (\lambda_j - \lambda_e' \overline{z}_j)d \overline{z}_j, \qquad \overline{\partial} q(Y) = i \sum_{j = 1}^n \epsilon_j y_j d \overline{z}_j.$$

\begin{lem}\thlabel{lem:pwedgep}
	We have the following formulas.
	\begin{enumerate}[(i)]
		\item
		$$\overline{*}\left(\overline{\partial} (\lambda, \psi(\overline{Z})) \wedge \overline{*} \overline{\partial} (\lambda, \psi(\overline{Z}))\right) = 2 (\lambda, \psi(Y))^2 - 4 q(Y) q(\lambda) + 4(\lambda, \psi(X)) q(Y) \lambda_e'.$$
		\item
		$$\overline{*}\ \frac{\lvert (\lambda, \psi(Z)) \rvert^2 \overline{\partial} q(Y) \wedge \overline{*} \overline{\partial} q(Y)}{q(Y)^2} = \lvert (\lambda, \psi(Z)) \rvert^2.$$
		\item
		$$-2 \Re\left(\overline{*}\ \frac{(\lambda, \psi(Z)) \overline{\partial} (\lambda, \psi(\overline{Z})) \wedge \overline{*} \overline{\partial} q(Y)}{q(Y)}\right) = -2(\lambda, \psi(Y))^2 - 4 (\lambda, \psi(X)) q(Y) \lambda_e'.$$
		\item
		\begin{align*}
			& \overline{*} \left(\frac{q(Y) \overline{\partial} (\lambda, \psi(\overline{Z})) - (\lambda, \psi(\overline{Z})) \overline{\partial} q(Y)}{q(Y)^2} \wedge \overline{*}\ \frac{q(Y) \overline{\partial} (\lambda, \psi(\overline{Z})) - (\lambda, \psi(\overline{Z})) \overline{\partial} q(Y)}{q(Y)^2}\right) \\
			&= -4 \frac{q(\lambda_{Z^-})}{q(Y)}.
		\end{align*}
	\end{enumerate}
\end{lem}

\begin{proof}
	Recall that $h^{ij}(Y) = 4 y_i y_j - 2 q(Y) \delta_{ij} \epsilon_i$, where $\delta_{ii} = 1$ and $\delta_{ij} = 0$ otherwise. Moreover, using \eqref{eq:dzWedgeHatdz} and \eqref{eq:Stardz} yields
	$$\overline{*} \left( d \overline{z}_i \wedge \overline{*} d \overline{z}_j \right) = \frac{1}{2} h^{ij}(Y).$$
	Therefore, we have
	\begin{align}
		\overline{*} \left(\left(\sum_{i = 1}^n f_i(Z) d \overline{z}_i\right) \wedge \overline{*} \left(\sum_{j = 1}^n g_j(Z) d \overline{z}_j\right) \right) = \frac{1}{2}\sum_{i, j = 1}^n f_i(Z) \overline{g_j(Z)} h^{ij}(Y). \label{eq:StarWedgeProdut}
	\end{align}
	and we will use this throughout the proof.
	
	We start with (i). Using \eqref{eq:StarWedgeProdut}, the left hand side is equal to
		\begin{align}
			&2 \sum_{i,j = 1}^n \epsilon_i \epsilon_j (\lambda_i - \overline{z}_i \lambda_e') (\lambda_j - z_j \lambda_e') h^{ij}(Y) \nonumber\\
			&= 8 \sum_{i,j = 1}^n \epsilon_i \epsilon_j (\lambda_i - \overline{z}_i \lambda_e') (\lambda_j - z_j \lambda_e') y_i y_j - 4 q(Y) \sum_{j = 1}^n \epsilon_j (\lambda_j - \overline{z}_j \lambda_e') (\lambda_j - z_j \lambda_e'). \label{eq:WeirdEq}
		\end{align}
		We use that
		$$\sum_{j = 1}^n \epsilon_j \lambda_j^2 = q(\lambda_W),$$
		where $\lambda_W$ denotes the projection of $\lambda$ to $W(\IR)$. Similarly, we have
		$$2 \sum_{j = 1}^n \epsilon_j \lambda_j z_j = (\lambda, Z) \qquad \text{and} \qquad 2 \sum_{j = 1}^n \epsilon_j \overline{z}_j z_j = (Z, \overline{Z}),$$
		so that \eqref{eq:WeirdEq} is equal to
		$$8 \left\lvert \sum_{i = 1}^n \epsilon_i (\lambda_i - z_i \lambda_e') y_i \right\rvert^2 - 2 q(Y) \Big(2q(\lambda_W) - (\lambda, Z) \lambda_e' - (\lambda, \overline{Z}) \lambda_e' + (Z, \overline{Z}) \lambda_e'^2\Big).$$
		Using
		$$2\sum_{i = 1}^n \epsilon_i (\lambda_i - z_i \lambda_e') y_i = (\lambda, Y) - (Z, Y) \lambda_e'$$
		and
		$$q(\lambda) = \lambda_e \lambda_e' + q(\lambda_W), \qquad (Z, \overline{Z}) = 2 q(X) + 2 q(Y), \qquad (\lambda, Z) + (\lambda, \overline{Z}) = 2(\lambda, X)$$
		we obtain
		\begin{align}
			2 \left\lvert (\lambda, Y) - (Z, Y) \lambda_e' \right\rvert^2 - 2 q(Y) \Big(2q(\lambda) - 2\lambda_e \lambda_e' - 2(\lambda, X) \lambda_e' + 2 q(X) \lambda_e'^2 + 2q(Y) \lambda_e'^2\Big). \label{eq:WeirdEq2}
		\end{align}
		Now,
		$$(\lambda, \psi(X)) = (\lambda, X) + \lambda_e + (q(Y) - q(X)) \lambda_e'$$
		and
		$$(\lambda, Y) - (Z, Y) \lambda_e' = (\lambda, \psi(Y)) - 2 i q(Y) \lambda_e'$$
		shows that \eqref{eq:WeirdEq2} is equal to
		\begin{align*}
			& 2 (\lambda, \psi(Y))^2 + 8 q(Y)^2 \lambda_e'^2 - 2 q(Y) \Big(2q(\lambda) - 2(\lambda, \psi(X)) \lambda_e' + 4 q(Y) \lambda_e'^2\Big) \\
			&= 2 (\lambda, \psi(Y))^2 - 4 q(Y) q(\lambda) + 4(\lambda, \psi(X)) q(Y) \lambda_e'.
		\end{align*}
		
		For (ii), we have
		\begin{align}
			\overline{*}\ \frac{\lvert (\lambda, \psi(Z)) \rvert^2 \overline{\partial} q(Y) \wedge \overline{*} \overline{\partial} q(Y)}{q(Y)^2} = \frac{\lvert (\lambda, \psi(Z)) \rvert^2}{2 q(Y)^2} \sum_{i,j = 1}^n \epsilon_i \epsilon_j y_i y_j h^{ij}(Y) \label{eq:iiEquation}
		\end{align}
		by \eqref{eq:StarWedgeProdut}. Now, a short calculation shows that \eqref{eq:iiEquation} is equal to
		\begin{align*}
			2\frac{\lvert (\lambda, \psi(Z)) \rvert^2}{q(Y)^2} \sum_{i,j = 1}^n \epsilon_i \epsilon_j y_i^2 y_j^2 - \frac{\lvert (\lambda, \psi(Z)) \rvert^2}{q(Y)} \sum_{j = 1}^n \epsilon_j y_j^2 = \lvert (\lambda, \psi(Z)) \rvert^2,
		\end{align*}
		where we used $q(Y) = \sum_{j = 1}^n \epsilon_j y_j^2$.
		
		For (iii) we again use \eqref{eq:StarWedgeProdut} to obtain
		\begin{align}
			&\overline{*}\ \frac{(\lambda, \psi(Z)) \overline{\partial} (\lambda, \psi(\overline{Z})) \wedge \overline{*} \overline{\partial} q(Y)}{q(Y)} \nonumber \\
			&= -i\frac{(\lambda, \psi(Z))}{q(Y)} \sum_{i,j} \epsilon_i \epsilon_j y_j (\lambda_i - \overline{z}_i \lambda_e') h^{ij}(Y) \nonumber \\
			&= -4i \frac{(\lambda, \psi(Z))}{q(Y)} \sum_{i,j} \epsilon_i \epsilon_j y_i y_j^2 (\lambda_i - \overline{z}_i \lambda_e') + 2i (\lambda, \psi(Z)) \sum_{j = 1}^n \epsilon_j y_j (\lambda_i - \overline{z}_i \lambda_e'). \label{eq:iiiEqu}
		\end{align}
		Using $q(Y) = \sum_{j = 1}^n \epsilon_j y_j^2$, we see that \eqref{eq:iiiEqu} is equal to
		\begin{align*}
			-2i (\lambda, \psi(Z)) \sum_{i = 1}^n \epsilon_i y_i (\lambda_i - \overline{z}_i \lambda_e')
			&= -i(\lambda, \psi(Z)) \Big((\lambda, Y) - (Y, \overline{Z}) \lambda_e'\Big) \\
			&= -i(\lambda, \psi(Z)) (\lambda, \psi(Y)) + 2 (\lambda, \psi(Z)) q(Y) \lambda_e',
		\end{align*}
		where we used
		$$2\sum_{j = 1}^n \epsilon_j y_j \lambda_j = (\lambda, Y) = (\lambda, \psi(Y)) + (X, Y) \lambda_e'$$
		and
		$$2 \sum_{j = 1}^n \epsilon_j y_j \overline{z}_j = (Y, \overline{Z}) = (Y, X) - 2 i q(Y).$$
		Now,
		\begin{align*}
			-2\Re\bigg(-i(\lambda, \psi(Z)) (\lambda, \psi(Y)) + 2 (\lambda, \psi(Z)) q(Y) \lambda_e'\bigg) = 2(\lambda, \psi(Y))^2 - 4 (\lambda, \psi(X)) q(Y) \lambda_e',
		\end{align*}
		which shows (iii).
		
		Now, the left hand side of (iv) is equal to
		\begin{align*}
			& \overline{*} \left(\frac{q(Y) \overline{\partial} (\lambda, \psi(\overline{Z})) - (\lambda, \psi(\overline{Z})) \overline{\partial} q(Y)}{q(Y)^2} \wedge \overline{*}\ \frac{q(Y) \overline{\partial} (\lambda, \psi(\overline{Z})) - (\lambda, \psi(\overline{Z})) \overline{\partial} q(Y)}{q(Y)^2}\right) \\
			&= q(Y)^{-2}\ \overline{*}\ \Bigg( \overline{\partial} (\lambda, \psi(\overline{Z})) \wedge \overline{*} (\lambda, \psi(\overline{Z})) + \frac{\lvert (\lambda, \psi(Z)) \rvert^2 \overline{\partial} q(Y) \wedge \overline{*} \overline{\partial} q(Y)}{q(Y)^2} \\
			&\qquad\qquad\qquad\qquad\qquad\qquad- 2 \Re\left(\overline{*}\ \frac{(\lambda, \psi(Z)) \overline{\partial} (\lambda, \psi(\overline{Z})) \wedge \overline{*} \overline{\partial} q(Y)}{q(Y)}\right) \Bigg).
		\end{align*}
		This is the sum of (i), (ii) and (iii) divided by $q(Y)^2$, hence, equal to
		\begin{align*}
			&q(Y)^{-2} \Bigg(2 (\lambda, \psi(Y))^2 - 4 q(Y) q(\lambda) + 4(\lambda, \psi(X)) q(Y) \lambda_e' + \lvert (\lambda, \psi(Z)) \rvert^2 \\
			&\qquad\qquad\qquad\qquad\qquad\qquad\qquad\ \,- 2(\lambda, \psi(Y))^2 - 4 (\lambda, \psi(X)) q(Y) \lambda_e' \Bigg) \\
			&= \frac{\lvert (\lambda, \psi(Z)) \rvert^2}{q(Y)^2} - \frac{4 q(\lambda)}{q(Y)} = -4 \frac{q(\lambda_{Z^-})}{q(Y)}.
		\end{align*}
		This proves the lemma.
\end{proof}

Let $p_Z(\lambda) = \xi_1 \frac{(\lambda, \psi(\overline{Z}))}{q(Y)}$. Then from the previous lemma we obtain

\begin{lem}\thlabel{lem:pWedgep}
	We have
$$\overline{*}_{-1} \left(\overline{\partial} \frac{(\lambda, \psi(\overline{Z}))}{q(Y)} \wedge p_Z(\lambda)\right) = -4 \frac{q(\lambda_{Z^-})}{q(Y)}.$$
\end{lem}

\begin{proof}
	Since $\overline{*}_\kappa = q(Y)^\kappa \overline{*},\ \xi_1 = q(Y) \overline{*} \overline{\partial}$ and $p_Z(\lambda) = q(Y) \overline{*} \overline{\partial} \frac{(\lambda, \psi(\overline{Z}))}{q(Y)}$, we have
	$$\overline{*}_{-1} \left(\overline{\partial} \frac{(\lambda, \psi(\overline{Z}))}{q(Y)} \wedge p_Z(\lambda)\right) = \overline{*} \left(\overline{\partial} \frac{(\lambda, \psi(\overline{Z}))}{q(Y)} \wedge \overline{*} \overline{\partial} \frac{(\lambda, \psi(\overline{Z}))}{q(Y)}\right).$$
	The quotient rule yields
	\begin{align*}
		\overline{\partial} \frac{(\lambda, \psi(\overline{Z}))}{q(Y)}
		&= \frac{q(Y) \overline{\partial} (\lambda, \psi(\overline{Z})) - (\lambda, \psi(\overline{Z})) \overline{\partial} q(Y)}{q(Y)^2}
	\end{align*}
	so that
	\begin{align*}
		&\overline{*} \left(\overline{\partial} \frac{(\lambda, \psi(\overline{Z}))}{q(Y)} \wedge \overline{*} \overline{\partial} \frac{(\lambda, \psi(\overline{Z}))}{q(Y)}\right) \\
		&= \overline{*} \left(\frac{q(Y) \overline{\partial} (\lambda, \psi(\overline{Z})) - (\lambda, \psi(\overline{Z})) \overline{\partial} q(Y)}{q(Y)^2} \wedge \overline{*}\ \frac{q(Y) \overline{\partial} (\lambda, \psi(\overline{Z})) - (\lambda, \psi(\overline{Z})) \overline{\partial} q(Y)}{q(Y)^2}\right).
	\end{align*}
	By \thref{lem:pwedgep}(iv) we obtain the result.
\end{proof}

The following lemma follows from \cite[Equation (2)]{Zemel} and the fact that the Laplace operator in \cite{Zemel} is $4 \Delta_\kappa$, see \cite[Remark 5.3]{Kiefer}.
\begin{lem}\thlabel{lem:LaplaceEigenfunction}
The function $\frac{(\lambda, \psi(\overline{Z}))}{q(Y)}$ is an eigenfunction with eigenvalue $\frac{n}{2}$ of the Laplace operator $\Delta_1$, i.e. $\Delta_1 \frac{(\lambda, \psi(\overline{Z}))}{q(Y)} = \frac{n}{2} \frac{(\lambda, \psi(\overline{Z}))}{q(Y)}$. In particular, we have $\xi_{-1} p_Z(\lambda) = \frac{n}{2} \frac{(\lambda, \psi(\overline{Z}))}{q(Y)}$.
\end{lem}

As a corollary we obtain
\begin{cor}\thlabel{cor:pIsHol}
	We have $\xi_{-1} \frac{p_Z(\lambda)}{q(\lambda_{Z^-})^{\frac{n}{2}}} = 0$.
\end{cor}

\begin{proof}
The quotient rule yields
\begin{align*}
	\xi_{-1} \frac{p_Z(\lambda)}{q(\lambda_{Z^-})^{\frac{n}{2}}}
	&= \frac{q(\lambda_{Z^-})^{\frac{n}{2}} \xi_{-1} p_Z(\lambda) - \frac{n}{2} q(\lambda_{Z^-})^{\frac{n}{2} - 1} \overline{*}_{-1} \left(\overline{\partial} q(\lambda_{Z^-}) \wedge p_Z(\lambda)\right)}{q(\lambda_{Z^-})^n}.
\end{align*}
Using $\overline{\partial} q(\lambda_{Z^-}) = -\frac{(\lambda, \psi(Z))}{4} \overline{\partial} \frac{(\lambda, \psi(\overline{Z}))}{q(Y)}$ together with \thref{lem:pWedgep} and \thref{lem:LaplaceEigenfunction}, we obtain that this expression equals
\begin{align*}
	\frac{\frac{n}{2}q(\lambda_{Z^-})^{\frac{n}{2}} \frac{(\lambda, \psi(\overline{Z}))}{q(Y)} - \frac{n}{2} q(\lambda_{Z^-})^{\frac{n}{2}}  \frac{(\lambda, \psi(\overline{Z}))}{q(Y)}}{q(\lambda_{Z^-})^n} = 0.
\end{align*}
\end{proof}

Let $F(a, b; c; z)$ be the principal branch of the Gauss hypergeometric function, see for example \cite[Chapter 15]{NIST}.

\begin{lem}\thlabel{lem:DifferentialEq}
	We have
	\begin{align*}
		&\overline{\partial} \left(q(\lambda_{Z^+})^{\frac{n}{2} - \kappa} F\left(1 - \frac{n}{2}, \kappa - \frac{n}{2}; \kappa - \frac{n}{2} + 1; \frac{q(\lambda)}{q(\lambda_{Z^+})}\right)\right) \\
		&= \frac{(-1)^{\frac{n}{2}}}{4} (\lambda, \psi(Z)) \left(\kappa - \frac{n}{2}\right) q(\lambda_{Z^+})^{-\kappa} q(\lambda_{Z^-})^{\frac{n}{2} - 1} \overline{\partial} \frac{(\lambda, \psi(\overline{Z}))}{q(Y)}
	\end{align*}
\end{lem}

\begin{proof}
 According to \cite[15.4.6]{NIST} we have
\begin{align*}
	F(a, b; b; z) = (1 - z)^{-a}
\end{align*}
and according to \cite[15.5.3]{NIST} we have
\begin{align*}
	\frac{\partial}{\partial z} z^{b} F(a, b; c; z) = b z^{b - 1} F(a, b + 1; c; z).
\end{align*}
In particular, the special value $F(a, b; b+1; z)$ satisfies
$$\frac{\partial}{\partial z} z^{b} F(a, b; b+1; z) = b z^{b - 1} F(a, b + 1; b + 1; z) = b z^{b - 1} (1 - z)^{-a}.$$
Using this together with $1 - \frac{q(\lambda)}{q(\lambda_{Z^+})} = - \frac{q(\lambda_{Z^-})}{q(\lambda_{Z^+})}$ and $q(\lambda_{Z^+}) = \frac{(\lambda, \psi(Z)) (\lambda, \psi(\overline{Z}))}{4 q(Y)}$ yields the result.
\end{proof}

\begin{defn}\thlabel{def:pForm}
	Define
	$$\tilde{p}_Z(\lambda) = \frac{(\lambda, \psi(Z))^{\kappa - 1}}{(-1)^{\frac{n}{2} - 1} 4^{\kappa} \left(\kappa - \frac{n}{2}\right)} \frac{p_Z(\lambda)}{q(\lambda_{Z^-})^{\frac{n}{2}}} q(\lambda_{Z^+})^{\frac{n}{2} - \kappa} F\left(1 - \frac{n}{2}, \kappa - \frac{n}{2}; \kappa - \frac{n}{2} + 1; \frac{q(\lambda)}{q(\lambda_{Z^+})}\right).$$
\end{defn}

\begin{rem}\thlabel{rem:pRewriting}
	The representation of $\tilde{p}_Z(\lambda)$ in \thref{def:pForm} is usually only useful if $q(\lambda) > 0$. If $q(\lambda) < 0$, we use \cite[15.8.1]{NIST} to rewrite it as
	$$\tilde{p}_Z(\lambda) = \frac{(\lambda, \psi(\overline{Z}))^{1-\kappa}}{(-1)^{\frac{n}{2} - 1} 4^\kappa (\kappa - \frac{n}{2}) q(Y)^{1-\kappa}} q(\lambda_{Z^-})^{\frac{n}{2} - 1} p_Z(\lambda)F\left(1 - \frac{n}{2}, 1; \kappa - \frac{n}{2} + 1; \frac{q(\lambda)}{q(\lambda_{Z^-})}\right).$$
	Although we will not need it here, we mention that for $q(\lambda) = 0$ we have $F(a, b; c; 0) = 1$ and thus
	$$\tilde{p}_Z(\lambda) = - \frac{p_Z(\lambda)}{\left(\kappa - \frac{n}{2}\right) (\lambda, \psi(Z))(\lambda, \psi(\overline{Z}))^{\kappa}}.$$
\end{rem}

\begin{rem}\thlabel{rem:ModularityTildeP}
	Since the function $\frac{(\lambda, \psi(\overline{Z}))}{q(Y)}$ is invariant under the weight $1$ slash operator with respect to the subgroup $\Gamma_\lambda$, the differential form $\tilde{p}_Z(\lambda)$ is invariant under the weight $-\kappa$ slash operator with respect to $\Gamma_\lambda$. I.e. it defines an element in $\calA^{n, n-1}(\Gamma_\lambda \bs \IH_n, \calL_{-\kappa})$ away from the singularity along $C_\lambda$ if $q(\lambda) > 0$ and $T_\lambda$ if $q(\lambda) < 0$. Moreover, it satisfies $\tilde{p}_Z(\lambda) \vert_{-\kappa} \gamma = \tilde{p}_Z(\gamma^{-1} \lambda)$ for all $\gamma \in G$.
\end{rem}

\begin{thm}\thlabel{thm:DifferentialEq}
	We have
	\begin{align*}
		\xi_{-\kappa} \tilde{p}_Z(\lambda) = \frac{1}{(\lambda, \psi(Z))^\kappa}.
	\end{align*}
\end{thm}

\begin{proof}
This is a direct consequence of \thref{lem:pWedgep}, \thref{cor:pIsHol} and \thref{lem:DifferentialEq}.
\end{proof}

\section{Locally Harmonic Maass Forms}\label{sec:LocallyHarmonicMF}

In this section we introduce locally harmonic Maass forms and prove that they satisfy a current equation involing cycle integrals along the real analytic cycles $C_\mu$. These are generalizations of \cite{BringmannLocallyHarmonic} (signature $(2, 1)$) and \cite[Section 5]{HilbertPaper} (signature $(2, 2)$) to general signature $(2, n)$.

\subsection{Locally Harmonic Maass Forms}

We start with the basic definition.

\begin{defn}
Let $\mu \in V$ with $q(\mu) > 0,\ \kappa > n$ and $\Gamma \subseteq O^+(L)$ a neat congruence subgroup. We define
$$\Omega_{\mu, -\kappa, \Gamma}^{\cusp}(Z) = \sum_{\gamma \in \Gamma_\mu \bs \Gamma} \tilde{p}_Z(\mu) \vert_{-\kappa} \gamma = \sum_{\lambda \in \Gamma \mu} \tilde{p}_Z(\lambda).$$
We will write $\Omega_\mu^{\cusp}$, if the weight $\kappa$ and the group $\Gamma$ are clear from the context. We denote the space generated by $\Omega_{\mu, -\kappa, \Gamma}^{\cusp},\ \mu \in V,\ q(\mu) > 0$ by $H_{-\kappa}^{\loc}(\Gamma)$ and call elements in $H_{-\kappa}^{\loc}(\Gamma)$ \emph{locally harmonic Maass forms}. Since the properties of \thref{lem:omegaProperties} also hold for $\Omega_\mu^{\cusp}$, we can define
$$C(\Omega^{\cusp}) = \sum_{\substack{\mu \in \Gamma \bs V / \IQ^\times \\ q(\mu) > 0}} c(\mu) \Gamma C_\mu,$$
if $\Omega^{\cusp} \in H_{-\kappa}^{\loc}(\Gamma)$ is given by
$$\Omega^{\cusp}(Z) = \sum_{\substack{\mu \in \Gamma \bs V / \IQ^\times \\ q(\mu) > 0}} c(\mu) q(\mu)^\kappa \Omega_{\mu, -\kappa, \Gamma}^{\cusp}.$$
\end{defn}

\begin{thm}\thlabel{thm:LocallyHarmonicMF}
\begin{enumerate}
	\item $\Omega_\mu^{\cusp}$ converges absolutely and locally uniformly away from its singularities along the $\Gamma$-translates of $C_\mu$ and defines an $(n, n-1)$-form with values in $\calL_{-\kappa}$.
\item We have $\xi_{-\kappa} \Omega_\mu^{\cusp} = \omega_\mu^{\cusp}$ away from the singularities. In particular, $\xi_{-\kappa} H_{-\kappa}^{\loc}(\Gamma)$ is the space generated by $\omega_\mu^{\cusp}$ and $\Omega_\mu^{\cusp}$ is harmonic outside the singularities.
\end{enumerate}
\end{thm}

\begin{proof}
(i) For convergence observe that it suffices to show that
	$$\left\lvert\left(\frac{q(\lambda_{Z^+})}{q(\lambda_{Z^-})}\right)^{\frac{n}{2}} \frac{p_Z(\lambda)}{(\lambda, \psi(Z))} \right\rvert$$
	is bounded for $\lambda \in \Gamma \mu$ and $Z \in K$, where $K$ is compact, so that the convergence follows from the convergence of $\omega_\mu^{\cusp}$. This follows from the fact, that for compact subsets $K \subseteq \IH_n$ the set
	$$\{ \lambda \in L' \ \vert\ q(\lambda) = m \text{ and there exists } Z \in \IH_n \text{ with } \pm q(\lambda_{Z^\pm}) < \varepsilon \}$$
	is finite.
	For the singularities observe that $(\mu, \psi(Z)) \neq 0$, since $q(\mu) > 0$, so that the singularities come from the $\Gamma$-translates of $q(\mu_{Z^-}) = 0$ and this is exactly the singularity along the $\Gamma$-translates of $C_\mu$. That this defines an $(n, n-1)$-form with values in $\calL_{-\kappa}$ follows from \thref{rem:ModularityTildeP}.
	
(ii) The first two assertions follow from \thref{thm:DifferentialEq}. For harmonicity use that $\omega_\mu^{\cusp}$ is holomorphic and $\Delta_{-\kappa} = \xi_\kappa \xi_{-\kappa}$ on $(n,n-1)$-forms.
\end{proof}

\begin{rem}\thlabel{rem:LocallyHarmonicMaassForm}
	Slightly more general one could define a locally harmonic Maass form as a harmonic $(n, n-1)$-form with values in $\calL_{-\kappa}$ of moderate growth that has the same type of singularity as a linear combination of $\Omega_\mu^{\cusp}$ for $\mu \in V, q(\mu) > 0$. This would coincide with the original definition in \cite{BringmannLocallyHarmonic}.
\end{rem}

\subsection{Integrals Around Real Analytic Cycles}

Let $\mu = b_1$, so that the real analytic cycle $C_\mu$ is given by
$$C_\mu = \{ (iy_1, x_2, \ldots, x_n)\ \vert\ y_1 > 0\} \subseteq \IH_n.$$
Let $0 < \varepsilon < 1$. Then the set
$$B_\varepsilon(C_\mu) = \left\{ (z_1, \ldots, z_n) \in \IH_n \ \bigg\vert\ x_1^2 < y_1^2 \varepsilon^2, \sum_{j = 2}^n y_j^2 < y_1^2 \varepsilon^2 \right\}$$
defines an open neighborhood of $C_\mu$, which is diffeomorphic to
$$C_\mu \times (-1, 1) \times B_{n-1},$$
where $B_{n-1} \subseteq \IR^{n-1}$ denotes the open ball of radius $1$ around $0$. The diffeomorphism is given by
$$\varphi_\varepsilon : ((iy_1', x_2', \ldots, x_n'), (x_1', y_2', \ldots, y_n')) \mapsto (y_1' x_1' \varepsilon + i y_1', x_2' + i y_1' y_2' \varepsilon, \ldots, x_n' + i y_1' y_n' \varepsilon).$$
We write $X' = (x_2', \ldots, x_n'),\ Y' = (y_2', \ldots, y_n')$ and $q'(Y') = y_2'^2 + \ldots + y_n'^2$.

For general $\mu \in V$ with $q(\mu) > 0$, we choose $\gamma \in G$ with $\sqrt{q(\mu)} \gamma b_1 = \mu$. Then $C_\mu = \gamma C_{b_1}$ and we define $B_\varepsilon(C_\mu) = \gamma B_\varepsilon(C_{b_1})$.

\begin{lem}\thlabel{lem:RealAnalCycleIntegral}
Let $\mu \in V$ with $q(\mu) > 0$ and let $h : \IH_n \to \IC$ be cuspidal of weight $\kappa$ for a neat congruence subgroup $\Gamma$ and smooth in $B_\varepsilon(C_\mu)$. Then
\begin{align*}
\lim_{\varepsilon \to 0} \int_{\Gamma_\mu \bs \partial B_\varepsilon(C_\mu)} h(Z) \tilde{p}_Z(\mu) = -C_{n, \kappa} q(\mu)^{\frac{n}{2}-\kappa} \int_{\Gamma_\mu \bs C_\mu} h(Z) (\mu, \psi(Z))^{\kappa - n} dZ,
\end{align*}
where
$$C_{n, \kappa} = \frac{(-i)^n \Gamma\left(\kappa - \frac{n}{2} + 1\right) \Gamma\left(\frac{n}{2}\right)}{4^{\kappa} \left(\kappa - \frac{n}{2}\right) \Gamma(\kappa)} (n-1) \vol(S_{n-2}) \int_{0}^1 \frac{r^{n-2}}{\sqrt{r^2 + 1}^n} dr.$$
\end{lem}

\begin{proof}
Both sides of the identity are homogeneous in $\mu$ of degree $n - 2 \kappa$, so that we can assume $q(\mu) = 1$. The case of general $\mu$ can then be reduced to $\mu = b_1$. Under the diffeomorphism $\varphi_\varepsilon$, we have
\begin{align*}
	\varphi_\varepsilon^* d x_1 &= \varepsilon (y_1' dx_1' + x_1' d y_1'), \qquad &\varphi_\varepsilon^* d x_j &= d x_j', \qquad j > 1, \\
	\varphi_\varepsilon^* d y_1 &= d y_1', \qquad &\varphi_\varepsilon^* d y_j &= \varepsilon (y_1' d y_j' + y_j' d y_1'), \qquad j > 1.
\end{align*}
Let
$$\widehat{d x_j} = d x_1 \wedge d y_1 \wedge \ldots \wedge d x_{j-1} \wedge d y_{j-1} \wedge d y_j \wedge \ldots \wedge d x_n \wedge d y_n$$
and analogously $\widehat{d y_j}$. Then
\begin{align*}
	\varphi_\varepsilon^* \widehat{d x_1} &= \varepsilon^{n-1} d y_1' \wedge \bigwedge_{j = 2}^n (d x_j \wedge (y_1' d y_j' + y_j' d y_1')) = (\varepsilon y_1')^{n - 1} \widehat{d x_1'} , \\
	\varphi_\varepsilon^* \widehat{d y_1} &= \varepsilon^{n} (y_1' d x_1' + x_1' d y_1') \wedge \bigwedge_{j = 2}^n (d x_j' \wedge (y_1' d y_j' + y_j' d y_1')), \\
	\varphi_\varepsilon^* \widehat{d x_i} &= \varepsilon^{n} (y_1' d x_1' + x_1' d y_1') \wedge d y_1' \wedge \bigwedge_{2 \leq j < i} (d x_j' \wedge (y_1' d y_j' + y_j' d y_1')) \\
	&\wedge (y_1' d y_i' + y_i' d y_1') \wedge \bigwedge_{n \geq j > i} (d x_j \wedge (y_1' d y_j' + y_j' d y_1')) = (\varepsilon y_1')^{n} \widehat{d x_i'}, \\
	\varphi_\varepsilon^* \widehat{d y_i} &= \varepsilon^{n-1} (y_1' d x_1' + x_1' d y_1') \wedge d y_1' \wedge \bigwedge_{2 \leq j < i} (d x_j' \wedge (y_1' d y_j' + y_j' d y_1')) \\
	&\wedge d x_i' \wedge \bigwedge_{n \geq j > i} (d x_j' \wedge (y_1' d y_j' + y_j' d y_1')) = (\varepsilon y_1')^{n-1} \widehat{d y_i'}.
\end{align*}
Hence, a short calculation shows that restricting to the boundary
$$\varphi_\varepsilon^{-1}(\partial B_\varepsilon(C_\mu)) = C_\mu \times (\{\pm 1 \} \times B_{n-1} \cup (-1, 1) \times S_{n-2}),$$
where $S_{n-2}$ denotes the boundary of $B_{n-1}$, yields
\begin{align*}
	&\lim_{\varepsilon \to 0} \frac{\varphi_\varepsilon^* \widehat{d x_1} \Big\vert_{\varphi_\varepsilon^{-1}(\partial B_\varepsilon(C_\mu))}}{\varepsilon^{n-1}} = y_1'^{n-1} \widehat{d x_1'}, \qquad &\lim_{\varepsilon \to 0} \frac{\varphi_\varepsilon^* \widehat{d y_1} \Big\vert_{\varphi_\varepsilon^{-1}(\partial B_\varepsilon(C_\mu))}}{\varepsilon^{n-1}} &= 0, \\
	&\lim_{\varepsilon \to 0} \frac{\varphi_\varepsilon^* \widehat{d x_i} \Big\vert_{\varphi_\varepsilon^{-1}(\partial B_\varepsilon(C_\mu))}}{\varepsilon^{n-1}} = 0, \qquad &\lim_{\varepsilon \to 0} \frac{\varphi_\varepsilon^* \widehat{d y_i} \Big\vert_{\varphi_\varepsilon^{-1}(\partial B_\varepsilon(C_\mu))}}{\varepsilon^{n-1}} &= y_1'^{n-1} \widehat{d y_i'}.
\end{align*}
Since,
$$\widehat{d \overline{z}_j} = i(2i)^{n-1} \widehat{d x_j} + (2i)^{n-1} \widehat{d y_j},$$
we obtain
\begin{align}
	\lim_{\varepsilon \to 0} \frac{\varphi_\varepsilon^* \widehat{d \overline{z}_j}}{\varepsilon^{n-1}} \Big\vert_{\varphi_\varepsilon^{-1}(\partial B_\varepsilon(C_\mu))} = \begin{cases}
		i (2i y_1')^{n-1} \widehat{d x_1'} &\mbox{ if } j = 1, \\
		(2i y_1')^{n - 1} \widehat{d y_j'} & \mbox{ if } j > 1.
	\end{cases} \label{eq:limitwidehatdz}
\end{align}
Next, since $(\mu, \psi(\overline{Z})) = 2 \overline{z}_1$, we have
\begin{align*}
	p_Z(\mu) = \overline{*}_1 \overline{\partial} \frac{2 \overline{z}_1}{q(Y)} = 2 \frac{q(Y)\ \overline{*} d \overline{z}_1 - z_1\ \overline{*} \overline{\partial} q(Y)}{q(Y)}.
\end{align*}
We have using \eqref{eq:Stardz}
$$\overline{*} \overline{\partial} q(Y) = -i \sum_{i = 1}^n \epsilon_i y_i\ \overline{*} d \overline{z}_i = \frac{i}{2(4i q(Y))^n} \sum_{i, j = 1}^n \epsilon_i y_i h^{ij}(Y)\ \widehat{d \overline{z}_j}.$$
Hence,
\begin{align*}
	p_Z(\mu)
	&= -\frac{q(Y) \sum_{j = 1}^n h^{1j}(Y) \widehat{d \overline{z}_j} + i z_1 \sum_{i, j = 1}^n \epsilon_i y_i h^{ij}(Y)\ \widehat{d \overline{z}_j}}{q(Y) (4i q(Y))^n} \\
	&= -\sum_{j = 1}^n \frac{q(Y) h^{1j}(Y) + i z_1 \sum_{i = 1}^n \epsilon_i y_i h^{ij}(Y)}{q(Y) (4i q(Y))^n} \widehat{d \overline{z}_j}.
\end{align*}
We use that $h^{ij}(Y) = 4 y_i y_j - 2q(Y) \delta_{ij} \epsilon_i$ and obtain
\begin{align}
	p_Z(\mu)
	&= -\sum_{j = 1}^n \frac{q(Y) (4 y_1 y_j - 2q(Y) \delta_{1j}) + i z_1 \sum_{i = 1}^n \epsilon_i y_i (4 y_i y_j - 2q(Y) \delta_{ij} \epsilon_i)}{q(Y) (4i q(Y))^n} \widehat{d \overline{z}_j} \nonumber \\
	&= -\sum_{j = 1}^n \frac{q(Y) (4 y_1 y_j - 2q(Y) \delta_{1j}) + i z_1 (4 y_j q(Y) - 2 q(Y) y_j)}{q(Y) (4i q(Y))^n} \widehat{d \overline{z}_j}, \label{eq:pZcalc}
\end{align}
where we used $\sum_{i = 1}^n \epsilon_i y_i^2 = q(Y)$. A straight forward calculation shows that \eqref{eq:pZcalc} is equal to
\begin{align*}
	&-\sum_{j = 1}^n \frac{4 y_1 y_j - 2q(Y) \delta_{1j} + 4 i z_1 y_j - 2i \epsilon_j z_1 y_j}{(4i q(Y))^n} \widehat{d \overline{z}_j} \\
	&= - 2\frac{(-q(Y) + i \overline{z}_1 y_1) \widehat{d \overline{z}_1} + i \overline{z}_1 \sum_{j = 2}^n y_j \widehat{d \overline{z}_j}}{(4i q(Y))^n},
\end{align*}
so that the pullback of $p_Z(\mu)$ is given by
\begin{align*}
	\varphi_\varepsilon^* p_Z(\mu) = -2\frac{y_1'^2 \varepsilon}{(4 i y_1'^2 (1 - \varepsilon^2 q'(Y')))^n} \left((\varepsilon q'(Y') + i x_1') \varphi_\varepsilon^*\widehat{d \overline{z}_1} + \sum_{j = 2}^n (i \varepsilon x_1' y_j' + y_j') \varphi_\varepsilon^* \widehat{d \overline{z}_j}\right).
\end{align*}
Therefore, using \eqref{eq:limitwidehatdz} yields
\begin{align}
	\lim_{\varepsilon \to 0} \frac{\varphi_\varepsilon^* p_Z(\mu) \Big\vert_{\varphi_\varepsilon^{-1}(\partial B_\varepsilon(C_\mu))}}{\varepsilon^n} = \frac{iy_1'^{1 - n}}{2^{n}} \left(- x_1' \widehat{d x_1'} + \sum_{j = 2}^n y_j' \widehat{d y_j'}\right). \label{eq:limitpZ}
\end{align}
Moreover, in these coordinates $q(\mu_{Z^-})$ is given by
$$\varphi_\varepsilon^* q(\mu_{Z^-}) = 1 - \frac{\varepsilon^2 y_1'^2 x_1'^2 + y_1'^2}{y_1'^2 - y_1'^2 \varepsilon^2 (y_2'^2 + \ldots + y_n'^2)} = -\varepsilon^2 \frac{q'(Y') + x_1'^2}{1 - \varepsilon^2 q'(Y')}.$$
So that for $\varepsilon \to 0$ using \eqref{eq:limitpZ} we obtain
\begin{align}
	\lim_{\varepsilon \to 0} \varphi_\varepsilon^* \frac{p_Z(\mu)}{q(\mu_{Z^-})^{\frac{n}{2}}} \Bigg\vert_{\varphi_\varepsilon^{-1}(\partial B_\varepsilon(C_\mu))} = \frac{iy_1'^{1 - n}}{(-1)^{\frac{n}{2}} 2^{n}} \left(\frac{- x_1' \widehat{d x_1'} + \sum_{j = 2}^n y_j' \widehat{d y_j'}}{(q'(Y') + x_1'^2)^{\frac{n}{2}}}\right). \label{eq:LimitOfpZ}
\end{align}
Using \cite[15.4.20]{NIST} shows
\begin{align}
	\lim_{\epsilon \to 0} \varphi^*_\varepsilon F\left(1 - \frac{n}{2}, \kappa - \frac{n}{2}; \kappa - \frac{n}{2} + 1; \frac{q(\mu)}{q(\mu_{Z^+})}\right) = \frac{\Gamma\left(\kappa - \frac{n}{2} + 1\right) \Gamma\left(\frac{n}{2}\right)}{\Gamma(\kappa)}. \label{eq:HyperGeomLimit}
\end{align}
Moreover, a direct calculation yields
\begin{align}
	\lim_{\varepsilon \to 0} \varphi^*_\varepsilon \left((\mu, \psi(Z))\right)^{\kappa - 1} = (2i y_1')^{\kappa - 1}, \qquad \qquad \lim_{\varepsilon \to 0} \varphi^*_\varepsilon q(\mu_{Z^+}) = 1. \label{eq:LimitOfqZ}
\end{align}
Collecting the terms from equations \eqref{eq:LimitOfpZ}, \eqref{eq:HyperGeomLimit} and \eqref{eq:LimitOfqZ}, we obtain
\begin{align*}
&\lim_{\varepsilon \to 0} \int_{\Gamma_\mu \bs \partial B_\varepsilon(C_\mu)} h(Z) \tilde{p}_Z(\mu) \\
&= -\frac{(-1)^{n} i^{\kappa}}{2^{\kappa + n + 1} \left(\kappa - \frac{n}{2}\right)} \frac{\Gamma\left(\kappa - \frac{n}{2} + 1\right) \Gamma\left(\frac{n}{2}\right)}{\Gamma(\kappa)} \int_{\Gamma_\mu \bs C_\mu} h(Z') y_1'^{\kappa - n} d y_1' d x_2' \ldots d x_n' \\
&\times \left(2 \int_{B_{n-1}} \frac{d y_2' \ldots d y_n'}{(q'(Y') + 1)^{\frac{n}{2}}} - \sum_{j = 2}^n \int_{-1}^1 \int_{S_{n-2}} y_j' \frac{d x_1' d y_2' \ldots d y_{j-1}' d y_{j+1}' \ldots d y_n'}{(1 + x_1'^2)^{\frac{n}{2}}} \right).
\end{align*}
The second summand vanishes by the symmetry of the function $y_j'$ on the sphere $S_{n-2}$, while the integral over $B_{n-1}$ is given by
$$(n-1) \vol(S_{n-2}) \int_{0}^1 \frac{r^{n-2}}{\sqrt{r^2 + 1}^n} dr,$$
so that this is equal to
\begin{align*}
	\lim_{\varepsilon \to 0} \int_{\Gamma_\mu \bs \partial B_\varepsilon(C_\mu)} h(Z) \tilde{p}_Z(\mu) = -C_{n, \kappa} \int_{\Gamma_\mu \bs C_\mu} h(\varphi(Z')) (2i y_1')^{\kappa - n} d y_1' d x_2' \ldots d x_n'.
\end{align*}
Now, the restriction of $dZ$ to $C_\mu$ is given by $d y_1' \wedge d x_2' \wedge \ldots \wedge d x_n'$ and $(\mu, \psi(Z)) = 2i y_1'$ on $C_\mu$, so that we obtain the result.
\end{proof}

\subsection{A Current Equation For Locally Harmonic Maass Forms}

Let $\Omega^{\cusp} \in H_\kappa^{\loc}(\Gamma)$ be a locally harmonic Maass form with corresponding real analytic cycle $C(\Omega^{\cusp})$. We can consider $\Omega^{\cusp}$ as a current on rapidly decreasing smooth forms $H \in \calA^{n, n-1}(\Gamma \bs \IH_n, \calL_{-\kappa})$ via
$$[\Omega^{\cusp}](H) = \int_{\Gamma \bs \IH_n} H \wedge \overline{*}_{-\kappa} \Omega^{\cusp}.$$
We extend the $\xi_{-\kappa}$ operator to currents via
$$(\xi_{-\kappa} [\Omega^{\cusp}])(h) = -[\Omega^{\cusp}](\xi_\kappa h),$$
where $h \in \calA^{0,0}(\Gamma \bs \IH_n, \calL_\kappa)$ is rapidly decreasing. Moreover, we define for $\mu \in V$ with $q(\mu) > 0$ the current
$$\delta_{C_\mu}(h) = q(\mu)^{\frac{n}{2} - \kappa} \int_{\Gamma_\mu \bs C_\mu} h(Z) (\mu, \psi(Z))^{\kappa - n} dZ.$$
By linearity we obtain currents $\delta_{C(\Omega^{\cusp})}$ for every $\Omega^{\cusp} \in H_{-\kappa}^{\loc}(\Gamma)$.

\begin{thm}\thlabel{thm:CurrentEq}
	We have
	$$\xi_{-\kappa} [\Omega^{\cusp}] = [\xi_{-\kappa} \Omega^{\cusp}] + C_{n, \kappa} \delta_{C(\Omega^{\cusp})}.$$
\end{thm}

\begin{proof}
	By linearity it suffices to prove this for $\Omega_\mu^{\cusp}$ with $\mu = b_1$. In this case we have
$$[\xi_{-\kappa} \Omega_\mu^{\cusp}] = [\omega_\mu^{\cusp}]$$
and we will calculate
$$\int_{\Gamma \bs \IH_n} h(Z) \overline{\omega_\mu^{\cusp}} q(Y)^\kappa d \mu(Z).$$
We unfold the integral against $\omega_\mu^{\cusp}$ and use \thref{thm:DifferentialEq} to obtain
$$\int_{\Gamma_\mu \bs \IH_n} h(Z) \overline{*}_\kappa \frac{1}{(\mu, \psi(Z))^\kappa} = \int_{\Gamma_\mu \bs \IH_n} h(Z) \overline{*}_\kappa \xi_{-\kappa} \tilde{p}_Z(\mu).$$
We would like to apply Stokes in the form of  \thref{lem:Stokes}, but $\tilde{p}_Z(\mu)$ is not smooth because of its singularity along $C_\mu$. Instead, we first cut out the neighborhood $B_\varepsilon(C_\mu)$, apply \thref{lem:Stokes} and then take the limit. This yields
\begin{align*}
&\lim_{\varepsilon \to 0} \int_{\Gamma_\mu \bs (\IH_n \setminus B_\varepsilon(C_\mu))} h(Z) \overline{*}_\kappa \xi_{-\kappa} \tilde{p}_Z(\mu) \\
&= - \lim_{\varepsilon \to 0} \int_{\Gamma_\mu \bs (\IH_n \setminus B_\varepsilon(C_\mu))} \xi_\kappa h(Z) \wedge \overline{*}_{-\kappa} \tilde{p}_Z(\mu) + \lim_{\varepsilon \to 0} \int_{\Gamma_\mu \bs \partial B_\varepsilon(C_\mu)} h(Z) \tilde{p}_Z(\mu).
\end{align*}
The first summand is equal to $\xi_{-\kappa}[\Omega_\mu^{\cusp}](h)$ by reversing the unfolding, while by \thref{lem:RealAnalCycleIntegral} the second summand is given by
$-C_{n,\kappa} \delta_{C(\Omega_\mu^{\cusp})}(h)$.
\end{proof}

\subsection{The Fourier Expansion of the Oda Lift}\label{sec:OdaLift}

As a corollary, we sketch how to recover the Fourier expansion of the Oda lift, compare \cite[Corollary 8.13]{HilbertPaper}, \cite[Lemma 8.7]{CrawfordFunke}. For simplicity we assume that $L$ is unimodular, i.e. $L' / L = \{ 0\}$, but the general case follows similarly, if one works with the Weil representation. We will write $e(x) = e^{2 \pi i x}$. Let $\Gamma(L) = \ker(O^+(L) \to O(L' / L))$ be the discriminant kernel and let $S_{\kappa}(\Gamma(L))$ be the space of holomorphic cusp forms of weight $\kappa$ with respect to $\Gamma(L)$, so that $\omega_\mu^{\cusp} \in S_\kappa(\Gamma(L))$ for all $\mu \in V$ with $q(\mu) > 0$. Following \cite[(5.12)]{Oda}, we consider for $Z \in \IH_n$ and $\tau \in \IH$ the generating series
\begin{align}
	\Omega(Z, \tau) = \sum_{\substack{\mu \in L' \\ q(\mu) > 0}} q(\mu)^{\kappa - \frac{n}{2}} \overline{\omega_\mu^{\cusp}(Z)} e(q(\mu) \tau). \label{eq:OmegaFourierExp}
\end{align}
If $S_{k}(\SL_2(\IZ))$ denotes the space of holomorphic cusp forms of weight $k = \kappa - \frac{n}{2} + 1$ with respect to $\SL_2(\IZ)$, then $\Omega(Z, \cdot) \in S_k(\SL_2(\IZ))$ with respect to the variable $\tau \in \IH$. The \emph{Oda lift} is then defined by
$$\Phi^{\Oda} : S_\kappa(\Gamma(L)) \to S_{k, L}(\SL_2(\IZ)), \qquad f \mapsto \int_{\Gamma(L) \bs \IH_n} f(Z) \Omega(Z, \tau) q(Y)^\kappa d \mu(Z).$$

\begin{cor}\thlabel{cor:OdaLift}
	If $f \in S_\kappa(\Gamma(L))$ is a cusp form, then the $m$th Fourier coefficient of $\Phi^{\Oda}(f)$ is given by
	$$-C_{n, \kappa} m^{\kappa - \frac{n}{2}} \sum_{\substack{\mu \in L' \\ q(\mu) = m}} \delta_{C_\mu}(f).$$
\end{cor}

\begin{proof}[Proof (Sketch)]
	The $m$th Fourier coefficient is equal to the Petersson inner product
	$$\int_{\SL_2(\IZ) \bs \IH} \Phi^{\Oda}(\tau) \overline{P_m(\tau)} \Im(\tau)^k d \mu(\tau),$$
	where $P_m$ is the (correctly normalized) holomorphic Poincar\'e series of weight $k$ and index $m$. By interchanging the two integrals, one ends up with
	$$\int_{\Gamma(L) \bs \IH_n} f(Z) \int_{\SL_2(\IZ) \bs \IH} \Omega(Z, \tau) \overline{P_m(\tau)} \Im(\tau)^k d \mu(\tau) q(Y)^\kappa d \mu(Z).$$
	The inner integral picks out the $m$th Fourier coefficient of $\Omega(Z, \tau)$, which is by \eqref{eq:OmegaFourierExp} equal to
	$$m^{\kappa - \frac{n}{2}} \sum_{\substack{\mu \in L' \\ q(\mu) = m}} \overline{\omega_\mu^{\cusp}(Z)}.$$
	Therefore, the $m$th Fourier coefficient of $\Phi^{\Oda}(\tau)$ is given by
	$$m^{\kappa - \frac{n}{2}} \sum_{\substack{\mu \in L' \\ q(\mu) = m}} \int_{\Gamma(L) \bs \IH_n} f(Z) \overline{\omega_\mu^{\cusp}(Z)} q(Y)^\kappa d \mu(Z) = m^{\kappa - \frac{n}{2}} \sum_{\substack{\mu \in L' \\ q(\mu) = m}} [\xi_{-\kappa} \Omega_\mu^{\cusp}](f).$$
	\thref{thm:CurrentEq} now yields the result, since $f$ is holomorphic.
\end{proof}

\section{Polar Harmonic Maass Forms}\label{sec:PolarHarmonicMF}

In this section we introduce polar harmonic Maass forms and show that their images under the $\xi_{-\kappa}$ operator are orthogonal to cusp forms. They are generalizations of \cite{BringmannPolar} (signature $(2, 1)$) and \cite[Section 5]{HilbertPaper} (signature $(2, 2)$) to arbitrary signature $(2, n)$.

\subsection{Polar Harmonic Maass Forms}

We start with the definition of polar harmonic Maass forms.

\begin{defn}
	Let $\nu \in V$ with $q(\nu) < 0$, $\kappa \in \IN$ and $\Gamma \subseteq O^+(L)$ a neat congruence subgroup. We define
	$$\Omega_{\nu, -\kappa, \Gamma}^{\mero}(Z) = \sum_{\gamma \in \Gamma_\nu \bs \Gamma} \tilde{p}_Z(\nu) \vert_{-\kappa} \gamma = \sum_{\lambda \in \Gamma \nu} \tilde{p}_Z(\lambda).$$
	As before, we will only write $\Omega_\nu^{\mero}$, if the weight $\kappa$ and the neat congruence subgroup $\Gamma$ is clear from the context. We denote the space generated by $\Omega_\nu^{\mero}, \nu \in V, q(\nu) < 0$ by $H_{-\kappa}^{\pol}(\Gamma)$ and call elements in $H_{-\kappa}^{\pol}(\Gamma)$ \emph{polar harmonic Maass forms}. The properties of \thref{lem:omegaProperties} are obviously also satisfied for $\Omega_\nu^{\mero}$.
\end{defn}

\begin{thm}\thlabel{thm:PolarHarmonicMF}
	\begin{enumerate}
		\item $\Omega_\nu^{\mero}$ converges absolutely and locally uniformly away from its singularities along the $\Gamma$-translates of $T_\nu$ and defines an $(n, n-1)$-form with values in $\calL_{-\kappa}$.
		\item We have $\xi_{-\kappa} \Omega_\nu^{\mero} = \omega_\nu^{\mero}$ away from the singularities. In particular, $\xi_{-\kappa} H_{-\kappa}^{\pol}(\Gamma)$ is the space generated by $\omega_\nu^{\mero}$ for $\nu \in V, q(\nu) < 0$ and $\Omega_\nu^{\mero}$ is harmonic outside the singularities.
	\end{enumerate}
\end{thm}

\begin{proof}
	(i) The convergence can be seen as in \thref{thm:LocallyHarmonicMF} using \thref{rem:pRewriting}. For the singularities first observe that $q(\nu_{Z^-}) \neq 0$, since $q(\nu) < 0$ so that the only possible singularities come from the $\Gamma$-translates of $q(\nu_{Z^+}) = 0$, i.e. are along the $\Gamma$-translates of $T_\nu$. Using the representation of \thref{rem:pRewriting}
		\begin{align*}
			\tilde{p}_Z(\nu) &= \frac{(\nu, \psi(\overline{Z}))^{1-\kappa}}{(-1)^{\frac{n}{2} - 1} 4^\kappa (\kappa - \frac{n}{2}) q(Y)^{1-\kappa}} q(\nu_{Z^-})^{\frac{n}{2} - 1} \\
			&\times p_Z(\nu)F\left(1 - \frac{n}{2}, 1; \kappa - \frac{n}{2} + 1; \frac{q(\nu)}{q(\nu_{Z^-})}\right),
		\end{align*}
		we see that $\Omega_\nu^{\mero}$ has in fact a singularity along the $\Gamma$-translates of $T_\nu$ coming from the $\Gamma$-translates of $(\nu, \psi(\overline{Z}))^{1-\kappa}$. That this defines an $(n, n-1)$-form with values in $\calL_{-\kappa}$ follows from \thref{rem:ModularityTildeP}.
	
	(ii) The first two assertions follow from \thref{thm:DifferentialEq}. Since $\omega_\mu^{\mero}$ is holomorphic outside the singularities and $\Delta_{-\kappa} = \xi_\kappa \xi_{-\kappa}$ on $(n,n-1)$-forms, we obtain that $\Omega_\nu^{\mero}$ is harmonic outside the singularities.
\end{proof}

\begin{rem}\thlabel{rem:PolarHarmonicMaassForms}
	As in \thref{rem:LocallyHarmonicMaassForm}, we could slightly generalize the notion of polar harmonic Maass forms as a harmonic $(n, n-1)$-form with values in $\calL_{-\kappa}$ of moderate growth that has the same type of singularity as a linear combination of $\Omega_\nu^{\mero}$ for $\nu \in V, q(\nu) < 0$. Moreover, in contrast to the more general definition of \cite{BringmannPolar}, we only allow singularities along the special divisors $T_\nu$.
\end{rem}

\subsection{Regularized Inner Products and an Orthogonality Result}

Let $\nu = b_n$, so that $q(\nu) < 0$. Then
$$T_\nu = \{ (Z', z_n) \in \IH_n\ \vert\ z_n = 0 \}.$$
Let $0 < \varepsilon < 1$. The subset
$$B_\varepsilon(T_\nu) = \{ (Z', z_n) \in \IH_n\ \vert\ \lvert z_n \rvert^2 < \varepsilon^2 q(Y') \}$$
is an open neighborhood of $T_\nu$. For general $\nu \in V$ with $q(\nu) < 0$ we choose $\gamma \in G$ with $\nu = \sqrt{q(\nu)} \gamma b_n$. Then $T_\nu = \gamma T_{b_n}$ and we define $B_\varepsilon(T_\nu) = \gamma B_\varepsilon(T_{b_n})$.

\begin{defn}\thlabel{def:RegularizedInnerProduct}
	We define the regularized inner product of a smooth cusp form $h : \IH \to \IC$ of weight $\kappa$ with $\omega_\nu^{\mero}$ by
	\begin{align*}
		\langle h, \omega_\nu^{\mero} \rangle^{\reg}
		&= \lim_{\varepsilon \to 0} \int_{\Gamma \bs (\IH_n \setminus B_\varepsilon(T_\nu))} h(Z) \wedge \overline{*}_\kappa \omega_\nu^{\mero}(Z) \\
		&= \lim_{\varepsilon \to 0} \int_{\Gamma \bs (\IH_n \setminus B_\varepsilon(T_\nu))} h(Z) \overline{\omega_\nu^{\mero}(Z)} q(Y)^\kappa d \mu(Z).
	\end{align*}
	This extends linearly to a regularized inner product between smooth cusp forms $h : \IH \to \IC$ and elements in the span of $\omega_\nu^{\mero}$ for $\nu \in V$ with $q(\nu) < 0$.
\end{defn}

\begin{lem}\thlabel{ref:AntiHolResidueThm}
	Let $H \in \calA^{n-1, n}(\Gamma_\nu \bs \IH_n, \calL_{-\kappa})$ be real analytic and cuspidal. Then for $\kappa \in \IN$ we have
	$$\lim_{\varepsilon \to 0} \int_{\Gamma_\nu \bs \partial B_\varepsilon(T_\nu)} \frac{H(Z)}{(\nu, \psi(Z))^\kappa} = 0.$$
\end{lem}

\begin{proof}
	It suffices to consider $\nu = b_n$. Locally, $H$ is given by
	$$\sum_{j = 1}^n H_j(Z) \widehat{d z_j}.$$
	The only term that plays a role when integrating is $H_n(Z) \widehat{d z_n}$ and we end up with integrals of the form
	$$\lim_{\varepsilon \to 0} \int_{U} \int_{\varepsilon \sqrt{q(Y)} \partial \ID} \frac{H_n(Z', z_n)}{(2 z_n)^\kappa} \widehat{d z_n \wedge d \overline{z}_n} \wedge d \overline{z}_n.$$ The inner integral vanishes as $\varepsilon \to 0$ by \cite[Lemma 8.2]{HilbertPaper}.
\end{proof}

\begin{thm}\thlabel{thm:OrthogonalityOnCF}
	For every real analytic cusp form $h : \IH_n \to \IC$ of weight $\kappa$ and $\omega^{\mero}$ in the span of $\omega_\nu^{\mero}, \nu \in V, q(\nu) < 0$, we have
	$$\langle h, \omega^{\mero} \rangle^{\reg} = 0.$$
\end{thm}

\begin{proof}
	By linearity it suffices to prove this for $\omega^{\mero} = \omega_\nu^{\mero}$ for $\nu \in V$ with $q(\nu) < 0$.	We use that by \thref{thm:PolarHarmonicMF} we have $\xi_{-\kappa} \Omega_\nu^{\mero} = \omega_\nu^{\mero}$, apply Stokes' theorem in the form of \thref{lem:Stokes} and unfold the integral to obtain
	\begin{align*}
		\langle F, \omega_\nu^{\mero} \rangle^{\reg}
		&= \lim_{\varepsilon \to 0} \int_{\Gamma_\nu \bs \partial B_\varepsilon(T_\nu)} \frac{h(Z) (\nu, \psi(\overline{Z}))^{\kappa - 1} \tilde{p}_Z(\mu)}{(\nu, \psi(\overline{Z}))^{\kappa - 1}}.
	\end{align*}
	Since $h(Z) (\nu, \psi(\overline{Z}))^{\kappa - 1} \tilde{p}_Z(\mu)$ is a real analytic $(n, n-1)$-form we can apply \thref{ref:AntiHolResidueThm} to the complex conjugate to see that this vanishes.
\end{proof}

\subsection{The Duality Theorem}

In this section we prove the duality \thref{thm:DualityTheorem} by combining the results of Section \ref{sec:LocallyHarmonicMF} and Section \ref{sec:PolarHarmonicMF}.

For $\nu = b_n$ and $H \in \calA^{n, n-1}(\Gamma \bs \IH_n, \calL_{-\kappa})$ we define the restriction
$$H \vert_{T_\nu}(Z') = \lim_{\varepsilon \to 0} \int_{\varepsilon \sqrt{q(Y)} \partial \ID} \frac{H(Z', z_n)}{(\nu, \psi(Z', z_n))^{\kappa}} \in \calA^{n-1, n-1}(\Gamma_\nu \bs T_\nu)$$
and the cycle integral
$$\delta_{T_\nu}(H) = \int_{\Gamma_\nu \bs T_\nu} H \vert_{T_\nu}.$$
In particular, we have
$$\lim_{\varepsilon \to 0} \int_{\Gamma_\nu \bs \partial B_{\varepsilon}(T_\nu)} \frac{H(Z)}{(\nu, \psi(Z))^\kappa} = \delta_{T_\nu}(H).$$
These definitions can be extended to arbitrary $\nu \in V$ with $q(\nu) < 0$ choosing $\gamma \in G$ with $\nu = \sqrt{q(\nu)} \gamma b_n$, so that $T_\nu = \gamma T_{b_n}$.

\begin{rem}
	If $H \in \calA^{n,n-1}(\Gamma \bs \IH_n, \calL_{-\kappa})$ is real analytic, then $H \vert_{T_\nu}$ is a higher order term of the power series expansion around the cycle $T_\nu$. We expect that one can calculate them using an extension of the higher pullbacks introduced by \cite{Williams}, compare \cite[Lemma 4.1]{Alfes}.
\end{rem}

We will now prove the main result of this paper.

\begin{thm}\thlabel{thm:DualityTheorem}
Let $\omega^{\mero}$ be in the span of $\omega_\nu^{\mero}$ for $q(\nu) < 0$ and let $\Omega^{\cusp} \in H_{-\kappa}^{\loc}(\Gamma)$ such that the cycles $C(\Omega^{\cusp})$ and $T(\omega^{\mero})$ do not intersect. Then
$$C_{n, \kappa} \delta_{C(\Omega^{\cusp})}(\omega^{\mero}) = \delta_{T(\omega^{\mero})}(\Omega^{\cusp}).$$
\end{thm}

\begin{proof}
	By linearity it suffices to consider $\omega_\nu^{\mero}, \Omega_\mu^{\cusp}$ for $q(\nu) < 0 < q(\mu)$ such that $C_\mu$ and $T_\nu$ do not intersect, i.e. $(\mu, \nu) \neq 0$. By \thref{thm:OrthogonalityOnCF} we have $\langle \omega_\mu^{\cusp}, \omega_\nu^{\mero} \rangle^{\reg} = 0$. Evaluating this differently using \thref{lem:Stokes} with the $\xi_{-\kappa}$-preimage $\Omega_\mu^{\cusp}$ yields on the other hand
	\begin{align*}
		\langle \omega_\mu^{\cusp}, \omega_\nu^{\mero} \rangle^{\reg} = \lim_{\varepsilon \to 0}\int_{\Gamma \bs \partial B_{\varepsilon}(C_\mu)} \omega_\nu^{\mero}(Z) \Omega_\mu^{\cusp}(Z) + \lim_{\varepsilon \to 0} \int_{\Gamma \bs \partial B_{\varepsilon}(T_\nu)} \omega_\nu^{\mero}(Z) \Omega_\mu^{\cusp}(Z),
	\end{align*}
	where we used that $\omega_\nu^{\mero}$ is holomorphic away from $T_\nu$. We unfold the first integral against $\Omega_\mu^{\cusp}$ and the second integral against $\omega_\nu^{\mero}$ to obtain
	\begin{align*}
		& \lim_{\varepsilon \to 0} \int_{\Gamma_\mu \bs \partial B_{\varepsilon}(C_\mu)} \omega_\nu^{\mero}(Z) \tilde{p}_Z(\mu) + \lim_{\varepsilon \to 0} \int_{\Gamma_\nu \bs \partial B_{\varepsilon}(T_\nu)} \frac{\Omega_\mu^{\cusp}(Z)}{(\lambda, \psi(Z))^\kappa} \\
		&= -C_{n,\kappa} \delta_{C_\mu}(\omega_\nu^{\mero}) + \delta_{T_\nu}(\Omega_\mu^{\cusp})
	\end{align*}
	by \thref{lem:RealAnalCycleIntegral}. This shows the result.
\end{proof}

\renewcommand\bibname{References}
\bibliographystyle{alphadin}
\bibliography{bibliography}

\end{document}

%% file: pakete.tex
\usepackage[english]{babel}

\usepackage{a4wide}

\usepackage[latin1]{inputenc}

\usepackage[T1]{fontenc}

\usepackage{enumerate}

\usepackage{amsmath, amssymb, amsthm}
\usepackage{mathrsfs}
\usepackage{mathtools}
\usepackage{mathabx}
\usepackage{scalerel}
\usepackage{theoremref}

\allowdisplaybreaks

\usepackage{imakeidx}
\makeindex[intoc]

\usepackage[pagebackref]{hyperref}
\numberwithin{equation}{section}